\documentclass{TAC}

\usepackage[usenames,dvipsnames]{pstricks}

\usepackage{amscd,amssymb,amsmath,graphicx,verbatim,hyperref}
\usepackage[OT1,T1]{fontenc}

\usepackage[all, cmtip]{xy}
\usepackage{pdfpages}

\usepackage{amsmath,amsfonts,mathabx,tikz,tikz-cd,stmaryrd,url,mathtools, mathrsfs, tensor, pigpen}
\usepackage{hyperref}
\hypersetup{
	colorlinks,
	citecolor=black,
	filecolor=black,
	linkcolor=black,
	urlcolor=black
}

\usetikzlibrary{matrix,arrows,decorations.pathmorphing}

\newcommand{\po}{\ar@{}[dr]|{\text{\pigpenfont R}}}
\newcommand{\pb}{\ar@{}[dr]|{\text{\pigpenfont J}}}
\newcommand{\fpb}{\ar@{}[dr]|{\text{fakepb}}}

\newcommand{\lra}{\longrightarrow}
\newcommand{\Ra}{\Rightarrow}

\newcommand{\ldual}[1]{\mathord{{\let\nolimits\relax\sideset{^\wedge}{}{#1}}}}
\newcommand{\laction}[2]{\mathord{{\let\nolimits\relax\sideset{^{#1}}{}{#2}}}}
\newcommand{\conj}[2]{\mathord{{\let\nolimits\relax\sideset{^{#1}}{}{#2}}}}

\newcommand{\xra}{\xrightarrow}
\newcommand{\xla}{\xleftarrow}

\newcommand{\xRa}{\xRightarrow}

\def\CE{{\mathscr E}}

\def\CK{{\mathscr K}}

\def\CM{{\mathscr M}}

\makeatletter
\newcommand*\bigcdot{\mathpalette\bigcdot@{.5}}
\newcommand*\bigcdot@[2]{\mathbin{\vcenter{\hbox{\scalebox{#2}{$\m@th#1\bullet$}}}}}
\makeatother

\newcommand{\twocong}[2][0.5]{\ar@{}[#2] \save ?(#1)*{\cong}\restore}
\newcommand{\twoeq}[2][0.5]{\ar@{}[#2] \save ?(#1)*{=}\restore}
\newcommand{\ltwocell}[3][0.5]{\ar@{}[#2] \ar@{=>}?(#1)+/r 0.2cm/;?(#1)+/l 0.2cm/^{#3}}
\newcommand{\rtwocell}[3][0.5]{\ar@{}[#2] \ar@{=>}?(#1)+/l 0.2cm/;?(#1)+/r 0.2cm/^{#3}}
\newcommand{\utwocell}[3][0.5]{\ar@{}[#2] \ar@{=>}?(#1)+/d  0.2cm/;?(#1)+/u 0.2cm/_{#3}}
\newcommand{\dtwocell}[3][0.5]{\ar@{}[#2] \ar@{=>}?(#1)+/u  0.2cm/;?(#1)+/d 0.2cm/^{#3}}
\newcommand{\ultwocell}[3][0.5]{\ar@{}[#2] \ar@{=>}?(#1)+/dr  0.2cm/;?(#1)+/ul 0.2cm/^{#3}}
\newcommand{\urtwocell}[3][0.5]{\ar@{}[#2] \ar@{=>}?(#1)+/dl  0.2cm/;?(#1)+/ur 0.2cm/^{#3}}
\newcommand{\dltwocell}[3][0.5]{\ar@{}[#2] \ar@{=>}?(#1)+/ur  0.2cm/;?(#1)+/dl 0.2cm/^{#3}}
\newcommand{\drtwocell}[3][0.5]{\ar@{}[#2] \ar@{=>}?(#1)+/ul  0.2cm/;?(#1)+/dr 0.2cm/^{#3}}

\begin{document}
\title{Variation on a comprehensive theme}
\author{Ross Street} 
\address{Centre of Australian Category Theory \\
Department of Mathematics and Statistics \\
Macquarie University, NSW 2109 \\ 
Australia}
\eaddress{ross.street@mq.edu.au} 
\thanks{The author gratefully acknowledges the support of Australian Research Council Discovery Grant DP190102432.}
\dedication{Dedicated to the memory of R.F.C. (Bob) Walters.}
\date{\small{\today}}
\maketitle

\begin{abstract}
\noindent The main result concerns a bicategorical factorization system on the bicategory $\mathrm{Cat}$ 
of categories and functors. Each functor $A\xra{f} B$ factors up to isomorphism as $A\xra{j}E\xra{p}B$
where $j$ is what we call an ultimate functor and $p$ is what we call a groupoid fibration.
Every right adjoint functor is ultimate. Functors whose ultimate factor is a right adjoint are shown to
have bearing on the theory of polynomial functors.
\end{abstract}

\tableofcontents

\section*{Introduction}

As an undergraduate I came across Russell \cite{Russell} and was quite disturbed by the state of foundations
for mathematics. 
The comprehension schema seemed central as a connection between mathematics and language.
Then I was happy with the breakthrough I saw in the papers \cite{Law1965, Law1969, Law1970} of Lawvere.

The factorization described here is an old idea I have been meaning to check thoroughly and write up but only now have found a reason to do so.
The reason relates to $\mathrm{Cat}$ as an example of a polynomic bicategory in the sense of my recent paper \cite{134}. We want to define a property of a functor in terms of one of its factors being special in some way.   
 
The idea for the present paper is a variant of the comprehensive factorization of a functor $A\xra{f} B$ as a
composite $A\xra{j} E\xra{p} B$ where $j$ is a final functor (in the sense of \cite{CWM} and used by Walters and the author in \cite{6} but sometimes called cofinal)
and $p$ is a discrete fibration. 
The name for the factorization system was chosen because of its relationship to the comprehension scheme for sets.
This is an orthogonal factorization system in the usual sense
on $\mathrm{Cat}$ as an ordinary category and in the enriched sense on $\mathrm{Cat}$ as
a (strict) 2-category. Here ``discrete'' means, of course, that the fibres of $p$ are sets. 

Now we wish to think of $\mathrm{Cat}$ as a bicategory and consider whether we obtain
a factorization system in a bicategorical sense when we behave totally bicategorically and
close our fibrations up under composition with equivalences and ask that the pseudofibres
be groupoids.

This works. Our proof models the proof of the usual comprehensive factorization 
as described by Verity and the author in \cite{104}. The final functors are replaced by 
what we call {\em ultimate}
functors and the discrete fibrations by what we call {\em groupoid fibrations}. In our
application, we are concerned with functors whose ultimate factor is a right adjoint.

I am grateful to Alexander Campbell
for pointing to the significantly related work of Joyal where $n$-final, $n$-fibration 
and homotopy factorization system are defined in the context of quasicategories; 
see page 170 of \cite{JoyV1} and Sections A.6-8 of \cite{JoyV2}.  

\section{Groupoid fibrations}

The following concept is called ``strongly cartesian'' by Grothendieck.
These morphisms are always closed under composition (unlike those he called ``cartesian'').

\begin{definition}\label{cartesianmor}
Let $p : E \to B$ be a functor. A morphism $\chi : e' \to e$ in $E$ is called {\em cartesian} for $p$ when the square \eqref{cart} is a pullback for all $k\in E$.
\begin{equation}\label{cart}
\begin{aligned}
\xymatrix{
E(k,e') \ar[rr]^-{E(k,\chi)} \ar[d]_-{p} && E(k,e) \ar[d]^-{p} \\
B(pk,pe') \ar[rr]_-{B(pk,p\chi)} && B(pk,pe)}
\end{aligned}
\end{equation}
\end{definition}

Since any commutative square with a pair of opposite sides invertible is a pullback, we see that all invertible morphisms in $E$ are cartesian, and, if $p$ is fully faithful, then all morphisms of $E$ are cartesian.

\begin{definition}\label{gpdfib}
The functor $p : E \to B$ is a {\em groupoid fibration} when 
\begin{itemize}
\item[(i)] for all $e\in E$
and $\beta : b \to pe$ in $B$, there exist $\chi : e'\to e$ in $E$
and invertible $b\cong pe'$ such that  $\beta = (b\cong pe' \xra{p\chi} pe)$, and
\item[(ii)] every morphism of $E$ is cartesian for $p$.  
\end{itemize}
\end{definition}

Our groupoid fibrations include all equivalences of categories and so are not necessarily fibrations in the sense of Grothendieck.

From the pullback \eqref{cart} it follows that groupoid fibrations are conservative (that is, reflect invertibility). So their pseudofibres $E_b$ are groupoids. 

For functors $A\xra{f}C\xla{g}B$, we write $f/g$ for the comma category (or slice) of $f$ and $g$; it is the top left 
vertex of a universal square
\begin{equation}\label{commasq}
\begin{aligned}
\xymatrix{
f/g \ar[d]_{s}^(0.5){\phantom{aaaa}}="1" \ar[rr]^{t}  && B \ar[d]^{g}_(0.5){\phantom{aaaa}}="2" \ar@{=>}"1";"2"^-{\lambda}
\\
A \ar[rr]_-{f} && C 
}
\end{aligned}
\end{equation}
in the bicategory $\mathrm{Cat}$.  
In particular, the arrow category of $E$ is $E^{\mathbf{2}} = 1_E/1_E = E/E$.
For a functor $E\xra{p}B$ and writing $B/p = 1_B/p$, there is a canonical functor
$E^{\mathbf{2}}\xra{r} B/p$ defined as follows.
$$
\xymatrix{
E^{\mathbf{2}} \ar@/_/[ddr]_{ps} \ar@/^/[drrr]^t
\ar@{.>}[dr]|-{r} \\
& B/p \ar[d]_{u}^(0.5){\phantom{aaaa}}="1" \ar[rr]^h
&& E \ar[d]^{p}_(0.5){\phantom{aaaa}}="2" \ar@{=>}"1";"2"^-{\lambda} \\
& B \ar[rr]_{1_B} && B } \\
\quad
\xymatrix{
 \\ & & \\
& \Huge{=} } \\
\quad
\xymatrix{
\\
E^{\mathbf{2}} \ar[d]_{p s}^(0.5){\phantom{aaaa}}="1" \ar[rr]^{t}  && B \ar[d]^{p}_(0.5){\phantom{aaaa}}="2" \ar@{=>}"1";"2"^-{p \lambda}
\\
B \ar[rr]_-{1_B} && B 
}
$$

We write $f/_{\mathrm{ps}}g$ for the full subcategory of the comma category $f/g$ of \eqref{commasq} consisting of those objects at which the component of $\lambda$ is invertible. It is called the 
{\em pseudopullback} or
{\em isocomma category} of the cospan $A\xra{f}C\xla{g}B$;
it is the top left vertex of a universal square
\begin{equation}\label{pspbsq}
\begin{aligned}
\xymatrix{
f/_{\mathrm{ps}}g \ar[d]_{s'}^(0.5){\phantom{aaaa}}="1" \ar[rr]^{t'}  && B \ar[d]^{g}_(0.5){\phantom{aaaa}}="2" \ar@{=>}"1";"2"^-{\lambda'}_{\cong}
\\
A \ar[rr]_-{f} && C 
}
\end{aligned}
\end{equation}
in the bicategory $\mathrm{Cat}$. 

Here are three fairly easy observations.

\begin{proposition}\label{3feo}
\begin{itemize}
\item[(a)] A functor $E\xra{p}B$ is a groupoid fibration if and only if the canonical $E^{\mathbf{2}}\xra{r} B/p$ is an equivalence.
\item[(b)] Suppose $E\xra{p}B$ is a groupoid fibration. A functor $F\xra{q}E$ is a groupoid fibration if and only if the composite $F\xra{q}E\xra{p}B$ is.
\item[(c)] The pseudopullback of a groupoid fibration along any functor is a groupoid fibration.
That is, if in \eqref{pspbsq} the functor $g$ is a groupoid fibration, so too is $s'$.
\end{itemize}
\end{proposition}

There is a 2-category $\mathrm{GFib}B$ of groupoid fibrations over $B$ defined as follows:
The objects are groupoid fibrations $E\xra{p}B$ over $B$. 
The hom categories are given by the pseudopullbacks:
 \begin{equation*}
 \begin{aligned}
\xymatrix{
\mathrm{GFib}B(p,q) \ar[d]_{}^(0.5){\phantom{AAAAAA}}="1" \ar[rr]^{}  && [E,F] \ar[d]^{[E,q]}_(0.5){\phantom{AAAAAA}}="2" \ar@{<=}"1";"2"^-{\cong}_-{}
\\
\mathbf{1} \ar[rr]_-{\lceil p\rceil} && [E,B] 
}
 \end{aligned}
\end{equation*}
So the morphisms are triangles with a natural isomorphism therein. 
\begin{equation}\label{morphoverB}
\begin{aligned}
\xymatrix{
E \ar[rd]_{p}^(0.5){\phantom{a}}="1" \ar[rr]^{f}  && F \ar[ld]^{q}_(0.5){\phantom{a}}="2" \ar@{=>}"1";"2"^-{\phi}_{\cong}
\\
& B 
}
\end{aligned}
\end{equation}

We also consider $\mathrm{Cat}/B$ with the same convention on its morphisms.

\section{Some fully faithful right adjoint functors}
\begin{equation}\label{Fffff}
\begin{aligned}
\xymatrix{
& & \mathrm{Ord} \ar[rd]^-{ \mathrm{incl}}  & &\\
\mathrm{Set}\ar[r]^-{\sim} &\mathrm{EqR} \ar[ru]^-{\mathrm{incl}} \ar[rd]_-{\mathrm{incl}} & & \mathrm{Cat}\ar[r]^-{\mathrm{nerve}} & [\Delta^{\mathrm{op}},\mathrm{Set}] \\
& & \mathrm{Gpd} \ar[ru]_-{\mathrm{incl}} & &}
\end{aligned}
\end{equation}

All the categories in the diagram \eqref{Fffff} are cartesian closed.
All the functors are ``closed under exponentiation''.
The left adjoints all preserve finite products (by Day Reflection Theorem).
Our focus here is on the inclusion $\mathrm{Gpd}\xra{\mathrm{incl}}\mathrm{Cat}$ with left 2-adjoint $\pi_1$ and right adjoint $\upsilon$.
The subcategory $\upsilon A$ of the category $A$ contains all and only the invertible morphisms of $A$. 

\begin{lemma}\label{upsfun}
A functor $E\xra{p}B$ is an equivalence if and only if both $\upsilon E\xra{\upsilon p}\upsilon B$ and $\upsilon (E^{\mathbf{2}})\xra{\upsilon (p^{\mathbf{2}})}\upsilon (B^{\mathbf{2}})$ are equivalences.
\end{lemma}
\begin{proof}
Only if is clear since $\upsilon$ is a 2-functor. 
For the converse first note that surjectivity on objects up to isomorphism for $p$ is the same as for $\upsilon p$.

So it remains to deduce from the groupoid equivalences that $p$ is fully faithful. 
Take $e, e'\in E$ and $pe \xra{\beta} pe'$ in $B$. Since $\upsilon (B^{\mathbf{2}})$ is surjective on objects up to isomorphism, there exists $e_1 \xra{\xi} e'_1$ in $E$ and a commutative square
\begin{equation*}
\xymatrix{
pe_1 \ar[r]^-{\sigma}_{\cong} \ar[d]_-{p\xi} & pe \ar[d]^-{\beta} \\
pe'_1 \ar[r]_-{\sigma'}^{\cong} & pe' \ .} 
\end{equation*}
Since $\upsilon p$ is full, there exist invertible $e_1 \xra{\chi} e$ and $e'_1 \xra{\chi'} e'$ in $E$
such that $p\chi = \sigma$ and $p\chi' = \sigma'$. Consequently, $\beta = p(\chi' \xi \chi^{-1})$ proving
that $p$ is full.

Since $\upsilon p$ is faithful, the only automorphisms in $E$ taken to identities by $p$ are identities.
We will use this special case in our proof now that $p$ is faithful. Take $\xi, \xi' : e\to e_1$ in $E$ 
with $p\xi = p\xi'$. Think of these two morphisms as objects of $E^{\mathbf{2}}$ which are taken
to two equal objects $p\xi, p\xi' : pe\to pe_1$ of $B^{\mathbf{2}}$. Since $p^{\mathbf{2}}$ is full,
the two objects $\xi$ and $\xi'$ are isomorphic by an isomorphism in $E^{\mathbf{2}}$ made up of 
automorphisms of $e$ and $e_1$ which are taken to identities by $p$.
Since those automorphisms must be identities, we deduce that $\xi=\xi'$, as required.    
\end{proof}

\begin{lemma}\label{upsgfib}
If $E\xra{p}B$ is a groupoid fibration and $\upsilon E\xra{\upsilon p}\upsilon B$ is an equivalence then $E\xra{p}B$ is an equivalence.
\end{lemma}
\begin{proof}
Since $\upsilon$ is a right adjoint, it preserves the pseudopullback
$$
\xymatrix{
E^{\mathbf{2}} \ar[rr]^-{\mathrm{cod}} \ar[d]_-{p^{\mathbf{2}}} && E \ar[d]^-{p} \\
B^{\mathbf{2}}\ar[rr]^-{\mathrm{cod}} && B}
$$
so that both $\upsilon p$ and $\upsilon (p^{\mathbf{2}})$ are equivalences.
The result follows by Lemma~\ref{upsfun}.
\end{proof}

\begin{proposition}
The usual ``Grothendieck construction'' 2-functor  
\begin{eqnarray*}
\wr  :  \mathrm{Hom}(B^{\mathrm{op}},\mathrm{Gpd}) \lra \mathrm{GFib}B
\end{eqnarray*}
is a biequivalence. If $\ \wr(T) \simeq (E\xra{p} B)$ then $Tb$ is equivalent to the pseudofibre $E_b$ of $p$ over $b\in B$.
\end{proposition}
The result of applying $\mathrm{Hom}(B^{\mathrm{op}},-)$ to the 2-adjunction 
$$ \xymatrix @C+5mm{
 \mathrm{Cat} \ar @<3pt> [r]^{\pi_1}  &  \mathrm{Gpd} \ar @<3pt>[l]^{\mathrm{incl}}}$$ 
transports to a biadjunction 
$$ \xymatrix @C+5mm{
 \mathrm{Fib}B \ar @<3pt> [r]^{\pi_{1 B}}  &  \mathrm{GFib}B \ar @<3pt>[l]^{\mathrm{incl}}
}$$
via the biequivalences
\begin{eqnarray*}
\mathrm{Hom}(B^{\mathrm{op}},\mathrm{Gpd}) \xra{\sim} \mathrm{GFib}B
\ \text{   and   }  \ \mathrm{Hom}(B^{\mathrm{op}},\mathrm{Cat})  \xra{\sim} \mathrm{Fib}B \ .
\end{eqnarray*}
\begin{remark}\label{overB} The inclusion 2-functor $\mathrm{GFib}B\hookrightarrow \mathrm{Cat}/B$ is fully faithful with a left biadjoint
whose value at the object $A\xra{f}B$ is the groupoid fibration $\pi_{1 B} (B/f\xra{\mathrm{dom}}B)$
which corresponds to the pseudofunctor $B^{\mathrm{op}}\to \mathrm{Gpd}$ taking $b\in B$ to $\pi_{1 } (b/f)$.
\end{remark}

The construction of $\pi_1A$ by generators and relations is awkward to work with; instead we use the following universal property of the coinverter construction. Write $[A,X]_{\cong}$ for the full subcategory of $[A,X]$ consisting of those functors $f : A\to X$ which invert all the morphisms of $A$. 
The adjunction unit $A\to \pi_1A$ induces an isomorphism 
$$[\pi_1A,X]\cong [A,X]_{\cong}$$
for all categories $X$ (not just groupoids).  

\section{Ultimate functors}
\begin{definition}
A functor $j : A \to B$ is called {\em ultimate} when, for all objects
$b\in B$, the fundamental groupoid $\pi_1(b/j)$ of the comma
category $b/j$ is equivalent to
the terminal groupoid:  
$$\pi_1(b/j) \simeq \mathbf{1} \ .$$
\end{definition}

\begin{proposition}\label{erafiu}
Every right adjoint functor is ultimate.
\end{proposition}
\begin{proof}
If $k\dashv j : A\to B$ then $b/j \simeq kb/A \to \mathbf{1}$ has a left adjoint owing to the initial object $1_{kb}$ of $kb/A$. Applying the 2-functor $\pi_1$ to the adjunction yields an adjunction between groupoids. 
\end{proof}
\begin{proposition}\label{pi1ufequiv}
Ultimate functors are taken by $\pi_1$ to equivalences.
\end{proposition}
\begin{proof}
Let $j:A\to B$ be ultimate. We must prove $\pi_1A\xra{\pi_1j}\pi_1B$ is an equivalence. 
What we prove is that, for any category $X$, 
if each diagonal functor $X\xra{\delta_{b}}[b/j,X]_{\cong}$ is an equivalence  then $[B,X]_{\cong}\xra{[j,1]_{\cong}}[A,X]_{\cong}$ is an equivalence.
Since $\delta_{b}$ is 2-natural in $b\in B$, any choice $\gamma_{b}$ of adjoint equivalence is pseudonatural: 
choose also counit $\varepsilon_b : \gamma_b\delta_b \xRa{\cong}1_X$ and unit $\eta_b : 1_A \xRa{\cong}\delta_b\gamma_b$. 
We will show that we have an inverse equivalence $\theta$ for $[j,1]_{\cong}$ defined by 
\begin{equation*}
\xymatrix{
\theta(f)b  \ar[d]_-{\theta(f)\beta} & = & \gamma_b(b/j\xra{\mathrm{cod}}A\xra{f}X \ar[d]^-{\gamma_{\beta, f\mathrm{cod}}}) \\
\theta(f)b'  & = & \gamma_b(b'/j\xra{\beta/j} b/j\xra{\mathrm{cod}}A\xra{f}X) \ .
}
\end{equation*}
For $g\in [B,X]_{\cong}$, we have isomorphisms
\begin{eqnarray*}
(\theta [j,1]_{\cong} g)b & = & \gamma_b(b/j\xra{\mathrm{cod}}A\xra{j}B\xra{g}X) \\
& \cong & \gamma_b(b/j\xra{!}1\xra{b}B\xra{g}X) \\
& \cong & \gamma_b\delta_b (gb) \\
& \xRa{\varepsilon_b \ \cong} & gb 
\end{eqnarray*}
naturally in $g$ and $b$, while, for $f\in [A,X]_{\cong}$, we have isomorphisms
\begin{eqnarray*}
([j,1]_{\cong} \theta) (f)a & = & \theta(f) ja \\
& \cong & \gamma_{ja}(ja/j\xra{\mathrm{cod}}A\xra{f}X) \\
&  \xRa{\eta^{-1} \ \cong} & (f\mathrm{cod})(ja\xra{1_{ja}}ja, a) \\
& = & f a 
\end{eqnarray*}  
naturally in $f$ and $a$. 
\end{proof}

\begin{proposition}\label{pspbultimate}
A functor is ultimate if and only if its pseudopullback along any (groupoid) opfibration is taken by $\pi_1$ to an equivalence. 
\end{proposition}
\begin{proof}
The pseudopullback $P\xra{\bar{j}} X$ of $A\xra{j} B$ along an opfibration $F\xra{q}B$ has $x/\bar{j}\simeq qx/j$; so $\bar{j}$ is ultimate if $j$ is. So $\pi_1$ takes $\bar{j}$ to an equivalence by Proposition~\ref{pi1ufequiv}. For the rest, in the pseudopullback
\begin{equation*}
\xymatrix{
b/j \ar[rr]^-{} \ar[d]_-{\mathrm{cod}} && b/B \ar[d]^-{\mathrm{cod}} \\
A \ar[rr]^-{j} && B \ ,}
\end{equation*}
note that $b/B$ has an initial object and $\mathrm{cod}$ is a groupoid opfibration.   
\end{proof}

\begin{proposition}
Every coinverter (localization) is ultimate.
\end{proposition}
\begin{proof}
Pullback along an opfibration has a right adjoint so coinverters are taken to coinverters. Also, $\pi_1$ takes coinverters to isomorphisms since it is a left adjoint and all 2-cells in $\mathrm{Gpd}$ are already invertible. Proposition~\ref{pspbultimate} applies.
\end{proof}

\begin{proposition}\label{3for2}
Suppose $A\xra{j} B$ is ultimate. A functor $B\xra{k} C$ is ultimate if and only if the composite $A\xra{j}B\xra{k}C$ is ultimate.
\end{proposition}
\begin{proof}
Look at the pasting
$$
\xymatrix{
Q \ar[r]^-{j'} \ar[d]_-{q''} & P \ar[r]^-{k'}  \ar[d]^-{q'}  & F \ar[d]^-{q} \\
A \ar[r]_-{j} & B \ar[r]_-{k} & C }
$$
of two pullbacks with $q$ a groupoid fibration. Since $j$ is ultimate, $j'$ is equivalenced by $\pi_1$.
So $k'j'$ is equivalenced by $\pi_1$ if and only if $k'$ is. 
\end{proof}

\begin{lemma}\label{fibreaspi1}
If $E\xra{p}B$ is a groupoid fibration and $X\xra{b}B$ is a functor from a groupoid $X$ then the composite $E_b\to b/p\to \pi_1(b/p)$ is an equivalence.
\end{lemma}
\begin{proof}
$E_b\to b/p$ is a left adjoint and $E_b$ is a groupoid already.
\end{proof}
\begin{proposition}\label{intersection}
Ultimate groupoid fibrations $E\xra{p} B$ are equivalences.
\end{proposition}
\begin{proof}
Let $\upsilon B\xra{b} B$ be the inclusion. Since $p$ is ultimate, the pullback $b/p\xra{g}b/B$ of $p$ along $b/B\xra{\mathrm{cod}}B$ is equivalenced by $\pi_1$. 
By Lemma~\ref{fibreaspi1},  $\pi_1(g)$ is equivalent
to $\upsilon (p)$. 
By Lemma~\ref{upsgfib}, since $p$ is a groupoid fibration with $\upsilon (p)$ an equivalence, $p$ is an equivalence. 
\end{proof}

\section{Bicategorical factorization systems}

The concept of factorization system in a bicategory is not new; for example, see \cite{61, DV}.
Before giving the definition, we revise the bicategorical variant of pullback.
\begin{definition}\label{defbipb} A square
\begin{equation}\label{bipbsq}
\begin{aligned}
\xymatrix{
W \ar[d]_{p}^(0.5){\phantom{aaaa}}="1" \ar[rr]^{q}  && B \ar[d]^{g}_(0.5){\phantom{aaaa}}="2" \ar@{=>}"1";"2"^-{\sigma}_{\cong}
\\
A \ar[rr]_-{f} && C 
}
\end{aligned}
\end{equation}
in a bicategory $\CK$ is called a {\em bipullback} of the cospan $A\xra{f}C\xla{g}B$ when, for all objects $K\in \CK$,
the functor $$\CK(K,W) \xra{(p,\sigma,q)}\CK(K,f)/_{\mathrm{ps}}\CK(K,g) \ ,$$ obtained from the universal property of the pseudopullback, is an equivalence. 
\end{definition}

\begin{remark}\label{bipbgpdfib} In the square \eqref{bipbsq}, if $g$ and $p$ are groupoid fibrations, then the square is a bipullback if and only if 
$$\upsilon(\CK(K,W)) \xra{\upsilon(p,\sigma,q)}\upsilon(\CK(K,f)/_{\mathrm{ps}}\CK(K,g))$$ 
is an equivalence of groupoids. This is because Proposition~\ref{3feo} (b) and (c) imply $(p,\sigma,q)$
is a groupoid fibration so that Lemma~\ref{upsgfib} applies.  
\end{remark}

A {\em factorization system on a bicategory} $\CK$ consists of a pair $(\CE,\CM)$ of sets $\CE$ and $\CM$ of morphisms of $\CK$ satisfying:
\begin{itemize}
\item[FS0.] if $f\cong mw$ with $m\in \CM$ and $w$ an equivalence then $f\in \CM$, while
if $f\cong we$ with $e\in \CE$ and $w$ an equivalence then $f\in \CE$;
\item[FS1.] for all $X\xra{e}Y\in \CE$ and $A\xra{m}B\in \CM$, the diagram 
\begin{equation}\label{psorthog}
\begin{aligned}
\xymatrix{
\CK(Y,A) \ar[d]_-{\CK(Y,m)}^(0.5){\phantom{aaa}}="1" \ar[rr]^-{\CK(e,A)}  &&  \CK(X,A) \ar[d]^-{\CK(X,m)}_(0.5){\phantom{aaa}}="2" \ar@{}"1";"2"^-{\cong}
\\
\CK(Y,B) \ar[rr]^-{\CK(e,B)} && \CK(X,B)}
\end{aligned}
\end{equation}
(in which the isomorphism has components of the associativity constraints for $\CK$) is a bipullback;
\item[FS2.] every morphism $f$ factorizes $f\cong m\circ e$ with $e\in \CE$ and $m\in \CM$. 
\end{itemize}

It follows that $\CE$ and $\CM$ are closed under composition and their intersection consists of precisely the equivalences. Moreover, in the square \eqref{psorthog}, the morphism $m$ is in $\CM$ if the square is a
bipullback for all $e\in \CE$, and dually. Also note that, if all morphisms in $\CM$ are groupoid fibrations
then Remark~\ref{bipbgpdfib} applies to simplify the bipullback verification for FS1.

\section{Main theorem}

\begin{theorem}\label{main}
Ultimate functors and groupoid fibrations form a bicategorical factorization system on $\mathrm{Cat}$. 
So every functor $f : A \to B$ factors pseudofunctorially 
as $f \cong (A\xra{j} E \xra{p} B)$ with $j$ ultimate and $p$ a groupoid fibration.
\end{theorem}
\begin{proof} FS0 is obvious. For FS2 construct the diagram
$$
\xymatrix{
A \ar[r]^-{i} \ar[d]_-{f} & B/f \ar[r]^-{n}  \ar[d]^-{\mathrm{dom}}  & E \ar[d]^-{p} \\
B \ar[r]^-{1} & B \ar[r]^-{1} & B }
$$
where $(E\xra{p} B) = \pi_{1B}(B/f \xra{\mathrm{dom}}B)$, the squares commute up to isomorphism, $i$ has a left adjoint $\mathrm{cod}$, and $n$ is a coinverter. 

It remains to prove FS1. By Remark~\ref{bipbgpdfib}, we must prove that, for any groupoid fibration $E\xra{p} C$ and any ultimate functor $A\xra{j} B$, the functor
$$([j,E], [B,p]) : [B,E] \lra [A,p]/_{\mathrm{ps}}[j,C]$$ 
is taken to an equivalence of groupoids by $\upsilon$.
By Remark~\ref{overB}, the value of the left biadjoint to $\mathrm{GFib}B\hookrightarrow \mathrm{Cat}/B$ at the ultimate functor $A\xra{j} B$ is equivalent to $B\xra{1_B} B$. So every morphism $j\xra{(f,\phi)}q$ over $B$
with $q$ a groupoid fibration factors up to isomorphism as
\begin{equation}\label{reflectultimate}
\begin{aligned}
\xymatrix{
j \ar[rd]_{(f,\phi)}^(0.5){\phantom{a}}="1" \ar[rr]^{(j,1_j)}  && 1_B \ar[ld]^{(w,\psi)}_(0.5){\phantom{a}}="2" \ar@{=>}"1";"2"^-{\sigma}_{\cong}
\\
& q 
}
\end{aligned}
\end{equation}
uniquely up to a unique isomorphism. In this, we have $f\xRa{\sigma}wj$ and $1_B\xRa{\psi}qw$
such that $\psi j = (j\xRa{\phi}qf\xRa{q\sigma}qwj)$.
Take any object $(u,\gamma ,v)$ of $[A,p]/_{\mathrm{ps}}[j,C]$; it consists of functors $A\xra{u}E, B\xra{v}C$
and an invertible natural transformation $pu\xRa{\gamma} vj$. By the universal property of the pseudopullback
$p/_{\mathrm{ps}}v$, the isomorphism $\gamma$ is equal to the pasted composite
$$
\xymatrix{
A \ar@/_/[ddr]_u \ar@/^/[drrr]^j
\ar@{.>}[dr]|-{u'} \\
& p/_{\mathrm{ps}}v \ar[d]_{s'}^(0.5){\phantom{aaaaa}}="1" \ar[rr]^{t'}
&& B \ar[d]^{v}_(0.5){\phantom{aaaaa}}="2" \ar@{=>}"1";"2"^-{\lambda'}_{\cong} \\
& E \ar[rr]_p && C \ . }
$$
By Proposition~\ref{3feo}, $p/_{\mathrm{ps}}v\xra{t'}B$ is a groupoid fibration. 
We can apply \eqref{reflectultimate} with $f = u'$, $q = t'$ and $\phi$ the identity of $j = t'u'$
to obtain $u'\xRa{\sigma}wj$ and $1_B\xRa{\psi}t'w$
such that $\psi j = (j=t'u'\xRa{t'\sigma}t'wj)$ uniquely up to a unique isomorphism of $(w,\psi,\sigma)$. 
This gives us $w'=s'w\in [B,E]$ and an isomorphism $(u,\gamma, v) \cong (w'j,1_{pw'j},pw') = ([j,E], [B,p])w'$
determined by the isomorphisms 
$$u=s'u'\xRa{s'\sigma}s'wj = w'j \ \text{ and } \ v\xRa{v\psi}vt'w = vt'w\xRa{(\lambda'w)^{-1}} ps'w=pw' \ .$$
This proves that the functor $\upsilon([j,E], [B,p])$ is surjective on objects up to isomorphism.
Now suppose we also have $h\in [B,E]$ and an isomorphism 
$$(\xi,\zeta) : ([j,E], [B,p])h\cong (w'j,1_{pw'j},pw')$$
which means we have invertible $hj\xRa{\xi}s'wj$ and $ph\xRa{\zeta}ps'w$ such that $p\xi = \zeta j$.
By the universal property of the pseudopullback, there exists a unique $k : B\to p/_{\mathrm{ps}v}$ such that
$s'k = h$, $t'k = 1_B$ and $\lambda'k = (ph\xRa{\zeta} ps'w\xRa{\lambda'w}vt'w\xRa{v\psi^{-1}}v)$, and there also exists a unique invertible
$\tau : kj\Ra wj$ such that $\xi = (hj=s'kj\xRa{s'\tau}s'wj)$ and $\psi j = (t'kj\xRa{t'\tau}t'wj)$.
So we have
\begin{equation*}
\begin{aligned}
\xymatrix{
j \ar[rd]_{(u',1_j)}^(0.5){\phantom{a}}="1" \ar[rr]^{(j,1_j)}  && 1_B \ar[ld]^{(k,1_{1_B})}_(0.5){\phantom{a}}="2" \ar@{=>}"1";"2"^-{\tau^{-1}\sigma}_{\cong}
\\
& t' 
}
\end{aligned}
\end{equation*} 
which allows us to use the uniqueness of $(w,\psi,\sigma)$ to obtain a unique isomorphism
$\kappa : k \Ra w$ such that $\kappa j = \tau$ and $t'\kappa = \psi$.
Then $\kappa' = s'\kappa : h\Ra w'$ is such that $\kappa' j = s'\kappa j = s' \tau = \xi$
and $p\kappa' = p s'\kappa = (\lambda'w)^{-1} (v\psi) (\lambda'k) = \zeta$.
Hence $\upsilon([j,E], [B,p])$ is full and it remains to prove it faithful.
So suppose we have an invertible $\delta' : h\Ra w'$ such that $\delta' j = \xi$
and $p\delta' = \zeta = (\lambda'w)^{-1} (v\psi) (\lambda'k)$. The universal property
of pseudopullback implies there exists $\delta : k\Ra w$ such that $s'\delta = \delta'$
and $t'\delta = \psi$, and implies we can deduce that $\delta j = \tau$ from the equations
$s'\delta j = \xi$ and $t'\delta j = \psi j = t'\tau$. By the uniqueness of $\kappa$, we have
$\delta = \kappa$ and hence $\delta' = \kappa'$, as required.   
\end{proof}

\section{Other possible variants}

It is possible that the factorization carries through for $(\infty,1)$-categories (also called quasicategories
or weak Kan complexes); see \cite{JoyV1, JoyV2}.
For the case of the tricategory 
$(2,1)\text{-}\mathrm{Cat}$ whose objects are bicategories with all 2-cells invertible, 
a basic ingredient would be the triadjunction
\[ \xymatrix @R-3mm {
(2,1)\text{-}\mathrm{Cat} \ar@<1.5ex>[rr]^{\pi_1} \ar@{}[rr]|-{\perp} &&(2,0)\text{-}\mathrm{Cat} \ar@<1.5ex>[ll]^{\mathrm{incl}}
}\]
where $(2,0)\text{-}\mathrm{Cat}$ is the subtricategory of $(2,1)\text{-}\mathrm{Cat}$ with all morphisms equivalences. There is an obvious core providing a right triadjoint too.
This requires the bumping up to factorization systems on tricategories. And, after all, as yet my application only needs the $\mathrm{Cat}$ case.

There is presumably also a version of the (ultimate, groupoid fibration) for categories internal to a category 
$\CE$ as done in \cite{104} for the usual comprehensive factorization.

Another direction concerns the laxer hierarchy of comprehension schema proposed by John Gray; see \cite{JWG99, JWG391}. What kinds of factorization do they provide?

\section{Application to polynomials}

In this section, we use our factorization to understand the implications of the paper \cite{134}
for polynomials in $\mathrm{Cat}$ as a bicategory.   

A morphism $p : E \to B$ in a bicategory is called a {\em groupoid fibration} when, for all
objects $A\in \CM$, the functor $\CM(A,p) : \CM(A,E)\to \CM(A,B)$ is a groupoid fibration as per
Definition~\ref{gpdfib}. 

A morphism $n :Y\to Z$ is called a {\em right lifter} when, for all $u : K\to Z$, there exists a right lifting of
of $u$ through $n$ (in the sense of \cite{12}).

Recall from \cite{134} that a bicategory $\CM$ with bipullbacks is always calibrated by the 
groupoid fibrations as the neat morphisms; that is, such a bicategory is polynomic. This
allows for the construction of a bicategory of ``polynomials'' in $\CM$. Indeed, Definition 8.2
of \cite{134} means for this situation that a {\em polynomial} $(m,S,p)$ from $X$ to $Y$ in $\CM$ is a span
$$X\xla{m}S\xra{p}Y$$ in $\CM$ with $m$ a right lifter and $p$ a groupoid fibration. 
To have a more explicit description we need to identify the right lifters in the given $\CM$.
 
\begin{proposition}
A functor is a right lifter in $\mathrm{Cat}$ if and only if it is a right adjoint.
\end{proposition}
\begin{proof}
Right adjoints in any bicategory are right lifters since the lifting is given by composing with the left adjoint.
Conversely, suppose the functor $Y\xra{n}Z$ is a right lifter. A right lift $1\xra{n_*(z)}Y$ for 
each object $1\xra{z}Z$ of $Z$ gives the components $nn_*(z) \xra{\epsilon_z}z$ of the counit of an adjunction $n_*\dashv n$; as in any book introducing adjoint functors, we know that the universal property of right lifter 
allows us to define $n_*$ on morphisms and so on. 
\end{proof}

In order to distinguish polynomials in the polynomic bicategory $\mathrm{Cat}$ from
polynomials in $\mathrm{Cat}$, in the sense of Weber \cite{Weber2015}, as a category with pullbacks, 
I use the term {\em abstract polynomial} for the former; that is, it is a span
$$A\xla{j_*} E \xra{p} B$$ 
of functors, where $p$ is a groupoid fibration and $j_*\dashv j$. 

A functor $f : A \to B$ is an {\em abstract polynomial functor} when, 
in its factorization $$f \cong (A\xra{j} E \xra{p} B)$$ as per Theorem~\ref{main},
the ultimate functor $j$ is a right adjoint. 

The next result follows from the work in \cite{134}; for convenience, we will include a direct proof. 

\begin{proposition}
Abstract polynomial functors compose. 
\end{proposition}
\begin{proof}
Take $A\xra{j} E \xra{p}B\xra{k} F \xra{q}C$ with $j_*\dashv j$, $k_*\dashv k$ and with $p,q$ groupoid fibrations. Form the pseudopullback
\begin{equation}\label{bipb}
 \begin{aligned}
\xymatrix{
P \ar[d]_{k'_*}^(0.5){\phantom{aaaa}}="1" \ar[rr]^{p'}  && F \ar[d]^{k_*}_(0.5){\phantom{aaaa}}="2" \ar@{<=}"1";"2"^-{\theta}_-{\cong}
\\
E \ar[rr]_-{p} && B 
}
 \end{aligned}
\end{equation}
to obtain the required ``distributive law''. One easily verifies there exists 
$k'_*\dashv k'$, $p'$ is a groupoid fibration and the Chevalley-Beck condition (as recalled on page 150 of \cite{9}) 
$$p'\circ k'\cong k\circ p$$ holds. So $q\circ k\circ p\circ j\cong q\circ p'\circ k'\circ j$ where
$q\circ p'$ is a groupoid fibration and $k'\circ j$ is a right adjoint.     
\end{proof}

Write $\mathrm{Cat}_{\mathrm{apf}}$ for the subcategory of $\mathrm{Cat}$ obtained by restricting the morphisms
to abstract polynomial functors. 

The next result is essentially Proposition 8.6 of \cite{134}.

\begin{proposition} 
If the bicategory $\CM$ is calibrated then, for each $K\in \CM$, there is a pseudofunctor
$\mathbb{H}_K : \mathrm{Poly}\CM \lra \mathrm{Cat}_{\mathrm{apf}}$ taking the polynomial $X\xla{m} S\xra{p} Y$
to the abstract polynomial functor which is the composite 
$$\CM(K,X)\xra{\mathrm{rif}(m,-)} \CM(K,S) \xra{\CM(K,p)} \CM(K,Y) $$
in $\mathrm{Cat}$. 
\end{proposition}
\begin{corollary} 
The pseudofunctor
$\mathbb{H}_{\mathbf{1}} : \mathrm{Poly}\mathrm{Cat} \lra \mathrm{Cat}_{\mathrm{apf}}$,
taking each abstract polynomial $A\xla{j_*} E \xra{p} B$ its associated abstract polynomial functor
$A\xra{j} E \xra{p} B$ with $j_*\dashv j$, is a biequivalence.
\end{corollary}

\begin{remark}
After my talk on this topic in the Workshop on Polynomial Functors 
\url{https://topos.site/p-func-2021-workshop/}, 
Paul Taylor kindly pointed out his 1988
preprint \cite{Tay1988} in which he distinguished parametric (or local) right adjoint functors with
motivation from proof theory and consequently calling them {\em stable} functors. 
His {\em trace} factorization for such a functor
is a right adjoint functor followed by a groupoid fibration. I am grateful to Clemens Berger for observing
that the groupoid fibrations so arising are a resticted class: their pseudofibres are coproducts
of codiscrete (chaotic) categories. However, it does show that every parametric right adjoint
functor provides an example of an abstract polynomial functor.
\end{remark}

\begin{center}
--------------------------------------------------------
\end{center}

\appendix


\begin{thebibliography}{000}

\bibitem{Ben1967} Jean B\'enabou, \textit{Introduction to bicategories}, Lecture Notes in Math. \textbf{47} (Springer, Berlin, 1967) 1--77.\label{Ben1967} 

\bibitem{61} Renato Betti, Dietmar Schumacher and Ross Street, \textit{Factorizations in bicategories} 
Dipartimento di Matematica, Politecnico di Milano, n. 22/R (May 1999); 
see 1997 draft at \url{http://science.mq.edu.au/~street/61.pdf}.\label{61}

\bibitem{DV} Magali Dupont and Enrico M. Vitale, \textit{Proper factorization systems in 2-categories}, 
J. Pure Appl. Algebra \textbf{179(1-2)} (2003) 65--86.\label{DV}

\bibitem{EW} Hans Ehrbar and Oswald Wyler, \textit{On subobjects and images in categories}, Report 68-34 (Carnegie Mellon University, November 1968) 28pp.\label{EW}

\bibitem{FreydKelly} Peter J. Freyd and G. Max Kelly, \textit{Categories of continuous functors I}, J. Pure and Applied Algebra \textbf{2} (1972) 169--191; Erratum Ibid. \textbf{4} (1974) 121.\label{FreydKelly}

\bibitem{JWG99} John W. Gray, \textit{The categorical comprehension scheme}, Lecture Notes in Math. \textbf{99} (Springer-Verlag, Berlin and New York, 1969) 242--312.\label{JWG99}

\bibitem{JWG391} John W. Gray, \textit{Formal category theory--Adjointness for 2-categories}, Lecture Notes in Math. \textbf{391} (Springer-Verlag, New York, 1974).\label{JWG391}

\bibitem{JoyV1} Andr\'e Joyal, \textit{Notes on quasi-categories} (University of Chicago, 22 June 2008); \url{https://www.math.uchicago.edu/~may/IMA/Joyal.pdf}.\label{JoyV1}

\bibitem{JoyV2} Andr\'e Joyal, \textit{The theory of quasi-categories and its applications, Volume 2};
\url{https://mat.uab.cat/~kock/crm/hocat/advanced-course/Quadern45-2.pdf}.\label{JoyV2}

\bibitem{Law1965} F.W. Lawvere, \textit{The category of categories as a foundation for mathematics}, Proceedings of the Conference on Categorical Algebra, La Jolla 1965 (Springer-Verlag, New York 1966) 1--20.\label{Law1965}

\bibitem{Law1969} F.W. Lawvere, \textit{Adjointness in Foundations}, Dialectica \textbf{23} (1969) 281--296.\label{Law1969}

\bibitem{Law1970} F.W. Lawvere, \textit{Equality in hyperdoctrines and comprehension schema as an adjoint functor}, Proceedings of the AMS Symposium on Pure Mathematics XVII (1970) 1--14.\label{Law1970}

\bibitem{CWM} Saunders Mac Lane, \textit{Categories for the Working Mathematician}, Graduate Texts in Mathematics \textbf{5} (Springer-Verlag, 1971).\label{CWM}

\bibitem{Russell} Bertrand Russell, \textit{The Principles of Mathematics} (Cambridge University Press, U/K. 1903).\label{Russell}

\bibitem{9} Ross Street, \textit{Elementary cosmoi I}, Lecture Notes in Math. \textbf{420} (Springer-Verlag, 1974) 134--180.\label{9}

\bibitem{14} Ross Street, \textit{Fibrations in bicategories}, Cahiers de topologie et g\'eom\'etrie diff\'erentielle \textbf{21} (1980) 111--160.\label{14}

\bibitem{134} Ross Street, \textit{Polynomials as spans}, Cahiers de topologie et g\'eom\'etrie diff\'erentielle cat\'egoriques \textbf{61(3)} (2020) 113--153.\label{134}

\bibitem{104} Ross Street and Dominic Verity, \textit{The comprehensive factorization and torsors}, Theory and Applications of Categories \textbf{23(3)} (2010) 42--75.\label{104}

\bibitem{6} Ross Street and R.F.C. Walters, \textit{The comprehensive factorization of a functor}, Bulletin American Math. Soc. \textbf{79} (1973) 936--941.\label{6}

\bibitem{12} Ross Street and Robert F.C. Walters, \textit{Yoneda structures on 2-categories}, J. Algebra \textbf{50} (1978) 350--379.\label{12}

\bibitem{Tay1988} Paul Taylor, \textit{The trace factorisation of stable functors}, Preprint (1988); available at \url{http://www.paultaylor.eu/stable/}.\label{Tay1988}

\bibitem{Weber2015} Mark Weber, \textit{Polynomials in categories with pullbacks}, Theory and Applications of Categories \textbf{30(16)} (2015) 533--598.\label{Weber2015}

\end{thebibliography}
\end{document}